\documentclass[11pt]{amsart}

\usepackage{amsmath,amssymb,amscd,amsfonts,graphicx,latexsym,epsfig,psfrag,url,color,xr,mathtools,enumitem,transparent,wasysym,float}
\usepackage{fullpage}
\usepackage[all]{xy}
\usepackage{hyperref}
\newtheorem{theorem}{Theorem}[section]

\newtheorem{proposition}[theorem]{Proposition}
\newtheorem{lemma}[theorem]{Lemma}

\theoremstyle{definition}
\newtheorem{definition}[theorem]{Definition}
\newtheorem{example}[theorem]{Example}
\theoremstyle{remark}
\newtheorem{remark}[theorem]{Remark}

\theoremstyle{plain}
\newcommand{\thistheoremname}{}
\newtheorem{genericthm}[theorem]{\thistheoremname}

\newtheorem*{genericthm*}{\thistheoremname}
\newenvironment{namedthm*}[1]
  {\renewcommand{\thistheoremname}{#1}%
   \begin{genericthm*}}
  {\end{genericthm*}}
  
\newcommand{\bC}{\mathbb{C}}

\newcommand{\bR}{\mathbb{R}}
\newcommand{\bZ}{\mathbb{Z}}

\newcommand\bm{\mathbf{m}}
\newcommand\bn{\mathbf{n}}
\newcommand\bzero{\mathbf0}

\newcommand{\on}{\operatorname}

\newcommand{\Fuk}{\on{Fuk}}
\newcommand{\comp}{C^2}
\newcommand{\Symp}{\mathsf{Symp}}

\renewcommand{\comp}{\text{comp}}
\newcommand{\seam}{\text{seam}}
\newcommand{\incom}{\text{in}}

\newcommand{\inte}{{\on{int}}}
\renewcommand{\root}{{\on{root}}}
\newcommand{\FM}{{\on{FM}}}
\newcommand{\GJ}{{\on{GJ}}}
\newcommand{\mr}{\mathring}
\newcommand{\Ass}{{\on{Ass}}}
\newcommand{\Top}{\textsf{Top}}
\renewcommand{\comp}{{\on{comp}}}
\newcommand{\mk}{{\on{mark}}}
\newcommand{\tree}{{\on{tree}}}

\newcommand\qu{/\kern-.7ex/} 
\newcommand\lqu{\backslash \kern-.7ex \backslash}

\newcommand{\ol}{\overline}

\newcommand{\sr}{\stackrel}

\newcommand{\wt}{\widetilde}

\newcounter{qcounter}

\newcommand\quotient[2]{
        \mathchoice
            {% \displaystyle
                \text{\raise1ex\hbox{$#1$}\Big/\lower1ex\hbox{$#2$}}%
            }
            {% \textstyle
                #1\,/\,#2
            }
            {% \scriptstyle
                #1\,/\,#2
            }
            {% \scriptscriptstyle
                #1\,/\,#2
            }
    }

\newcommand\quoti[2]{
                \text{\raise1ex\hbox{$#1$}/\lower1ex\hbox{$\scriptstyle#2$}}
  }

\newcommand\quot[2]{
                \text{\raise1ex\hbox{$#1\!\!$}/\lower1ex\hbox{$\!\scriptstyle#2$}}
  }

\newcommand\quo[2]{
                \text{\raise.8ex\hbox{$\scriptstyle#1\!$}/\lower.8ex\hbox{$\!\scriptstyle#2$}}
  }

\newcommand\qq[2]{
                \text{\raise.8ex\hbox{$#1\!$}/\lower.8ex\hbox{$#2$}}
}

\begin{document}

\title{A simplicial version of the 2-dimensional Fulton--MacPherson operad}
\author{Nathaniel Bottman}
%\address{Department of Mathematics, University of Southern California,
%3620 S Vermont Ave, Los Angeles, CA 90089, USA}
%\address{Max Planck Institute for Mathematics,
%Vivatsgasse 7, 53111 Bonn, Germany}
%\email{\href{mailto:bottman@mpim-bonn.mpg.de}{bottman@mpim-bonn.mpg.de}}

\maketitle

\begin{abstract}
We define an operad in \textsf{Top}, called $\FM_2^W$.
The spaces in $\FM_2^W$ come with CW decompositions, such that the operad compositions are cellular.
In fact, each space in $\FM_2^W$ is the realization of a simplicial set.
We expect, but do not prove here, that $\FM_2^W$ is isomorphic to the 2-dimensional Fulton--MacPherson operad $\FM_2$.
Our construction is connected to the author's work on the symplectic $(A_\infty,2)$-category, and suggests a strategy toward equipping the symplectic cochain complex with the structure of a homotopy Batalin--Vilkoviskiy algebra.
\end{abstract}

\section{Introduction}
\label{s:intro}

In 1994, Getzler--Jones \cite{getzler_jones} introduced the Fulton--MacPherson operad
\begin{align}
\FM_2 = \bigl(\FM_2(k)\bigr)_{k \geq 1},
\end{align}
where $\FM_2(k)$ is the compactification \`a la Fulton--MacPherson \cite{fm} of the configuration space of $k$ distinct labeled points in $\bR^2$, modulo translations and dilations.
Getzler and Jones proposed in the same paper a collection of cellular decompositions of the spaces in $\FM_2$, such that these decompositions are compatible with the operad maps $\circ_i\colon \FM_2(k) \times \FM_2(\ell) \to \FM_2(k+\ell-1)$.
These decompositions formed the basis for a significant amount of work related to the Deligne conjecture, including a proof in \cite{getzler_jones} of that conjecture.

Unfortunately, Tamarkin found an error in Getzler--Jones' decomposition.
In particular, in the 9-dimensional space $\FM_2(6)$, there are two disjoint open 6-cells $C_1, C_2$ with the property that $\ol{C_1} \cap C_2$ is nonempty, as described in \cite[\S1.2.2]{voronov}.
In a recent preprint \cite{salvatore_main}, Salvatore used meromorphic differentials to construct cellular decompositions of the spaces in $\FM$.
His approach is completely different from Getzler--Jones'.

In this paper, we construct an operad of CW complexes, which we conjecture to be isomorphic in $\Top$ to $\FM_2$.
Under this expected isomorphism, our decompositions are refinements of Getzler--Jones' attempted decompositions.
The context for the current paper is the author's program (as developed in \cite{b:2-associahedra,b:realization,b:sing,b:thesis,bc,bo,bw:compactness}) to construct $\Symp$, the symplectic $(A_\infty,2)$-category.
Specifically, the author plans to use the decompositions of $\FM$ that we construct here to understand the axioms for identity 1-morphisms in an $(A_\infty,2)$-category.
In the context of $\Symp$, this suggests a strategy toward endowing symplectic cohomology with a chain-level homotopy Gerstenhaber (and eventually, homotopy BV) algebra structure that is finite in each arity, thus answering Conjecture 2.6.1 from \cite{abouzaid}.
We note that our approach is compatible with the operations in $\Symp$, unlike Salvatore's; in addition, we expect our approach to generalize to the Fulton--MacPherson operad of any dimension.

\subsection{Getzler--Jones' attempted decomposition}

Getzer--Jones' attempted decomposition is an adaptation to the case of $\FM_2$ of Fox--Neuwirth's decomposition \cite{fox_neuwirth} of the one-point compactification of the configuration space $(\bR^2)^k \setminus \Delta$ of $k$ points in $\bR^2$, where $\Delta$ is the fat diagonal.
A Fox--Neuwirth cell corresponds to a choice of which subsets of the points $p_1, \ldots, p_k$ should be vertically aligned, the left-to-right order in which these subsets of points should appear, and the top-to-bottom order in which each subset of the points should appear.
For instance, the following is a real-codimension-3 cell in $\bigl((\bR^2)^6\setminus\Delta\bigr)^*$:
\begin{figure}[H]
\centering
\def\svgwidth{0.175\columnwidth}
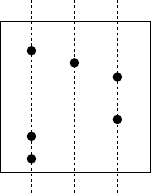
\caption{\label{fig:GJ}}
\end{figure}
\noindent
Getzler--Jones observed that the Fox--Neuwirth cells are invariant under translations and dilations, and moreover that one can define a similar type of cell for the boundary locus.
The elements in the boundary of $\FM_2(k)$ are trees of ``screens'', and these ``boundary cells'' are defined by partitioning and ordering the points on each of the screen in the same way as with Fox--Neuwirth cells.

\subsection{Tamarkin's counterexample}
\label{ss:tamarkin}

As described in \cite{voronov}, Tamarkin observed a way in which Getzler--Jones' supposed decomposition fails.
Consider $\FM_2(6)$, the open locus of which parametrizes configurations of six distinct points in $\bR^2$, up to translations and dilations.
Next, we consider the two 6-cells $C_1$ and $C_2$ pictured below.
(We omit the numberings.)
\begin{figure}[H]
\centering
\includegraphics[width=0.5\textwidth]{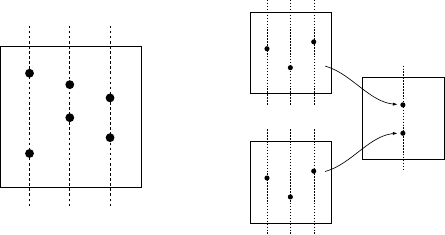}
\end{figure}
\noindent
The $j$-th bubble in $C_2$ (for $j=1,2$) carries a modulus $\lambda_j$ defined in the following way: by translating and dilating, we can move the left resp.\ right lines to $x=0$ and $x=1$; we then denote by $\lambda_j$ the position of the middle line.
The intersection $\ol{C_1} \cap C_2$ is the codimension-1 locus in $C_2$ in which $\lambda_1 = \lambda_2$.
What Getzler--Jones proposed is therefore not a cellular decomposition, because the intersection of the closures of two distinct $n$ cells should be contained in the $(n-1)$-skeleton.

In our construction, $C_1$, $C_2$, and $\ol{C_1} \cap C_2$ will each be a union of cells.

\subsection{An overview of our construction}
\label{ss:overview}

In this paper, we construct a collection of CW complexes $\FM_2^W(k)$ and maps
\begin{align}
\circ_i
\colon
\FM_2^W(k) \times \FM_2^W(\ell)
\to
\FM_2^W(k+\ell-1),
\qquad
1 \leq i \leq k.
\end{align}
Here is our main result:

\medskip

\noindent
{\bf Main theorem.}
{\it The spaces $\bigl(\FM_2^W(k)\bigr)_{k\geq1}$ together with the composition operations $\circ_i$ form a non-$\Sigma$ operad, and the composition maps
\begin{align}
\circ_i
\colon
\FM_2^W(k) \times \FM_2^W(\ell)
\to
\FM_2^W(k+\ell-1)
\end{align}
are cellular.}

\bigskip

We will now give a brief overview of the definition of $\FM_2^W(k)$.

\bigskip

\noindent
{\bf 1.}
First, we define a ``W-version'' $W_\bn^W$ of the 2-associahedra by completing the following analogy:
\begin{align}
K_r \quad:\quad W(\Ass) \quad::\quad W_\bn \quad:\quad W_\bn^W.
\end{align}
Here $K_r$ is the $(r-2)$-dimensional associahedron, and $W(\Ass)$ is the Boardman--Vogt W-construction applied to the associative operad, which is defined in terms of metric stable trees and yields an operad of CW complexes that is isomorphic to the associahedral operad $K$ in $\Top$.
$W_\bn$ is an $(|\bn|+r-3)$-dimensional 2-associahedron, and $W_\bn^W$ is a CW complex that we define in \S\ref{s:W_n^W} in terms of metric stable tree-pairs and which we expect to be homeomorphic to $W_\bn$.
We then refine the CW structure on $W_\bn^W$ to a simplicial decomposition.

\medskip

\noindent
{\bf 2.}
Toward our construction of $\FM_2^W(k)$, we decompose $\FM_2(k)$ into Getzler--Jones cells, then identify each open Getzler--Jones cell with a product of open 2-associahedra.
We then replace each such product by the corresponding product of interiors of the spaces $W_\bn^W$ described in the previous step.
This product comes with a decomposition into products of simplices, and we refine this to a simplicial structure.
Finally, we attach these decomposed Getzler--Jones cells together to produce $\FM_2^W(k)$.
This part of the construction appears in \S\ref{s:FM_2^W}.

\bigskip

The essential property of $\FM_2^W(k)$ that we must verify is that our CW decomposition is valid.
It is clear that our putative open cells disjointly decompose our space, and that they are homeomorphic to open balls.
The only nontrivial check we need to make is that the $n$-cells are attached to the $(n-1)$ skeleton.
This is where Getzler--Jones' attempted decomposition fails: the 6-cell $C_1$ that we described in \S\ref{ss:tamarkin} is not attached to the 5-skeleton.
Our decomposition satisfies this property by construction: we attach a given $n$-cell by taking a closed $n$-simplex, then attaching it to the existing skeleton via quotient maps from the boundary $(n-1)$-simplices to the $(n-1)$-skeleton.
In fact, the boundary of an $n$-cell is a union of cells of dimension at most $n-1$.

\subsection{The relationship between our construction and $\Symp$}
\label{ss:relationship}

%Fix $M \geq 1$ and a partition of $\{1,\ldots,M\}$ into disjoint subsets $A_1,\ldots,A_k$, where some of the $A_i$'s may be empty.
%Fix a linear order on each $A_i$.

The genesis of the construction of $\FM_2^W$ was a connection between the symplectic $(A_\infty,2)$-category $\Symp$ and $E_2$ suggested by Jacob Lurie in 2016.
(The construction of $\Symp$ is a long-term project of the author, building on work of Ma'u--Wehrheim--Woodward; see \cite{b:2-associahedra,b:realization,b:sing,b:thesis,bc,bw:compactness,mww}.)
We can express this connection concretely, via a collection of maps
\begin{align}
f_\sigma^W\colon W_\bn^W \to \FM_2^W(|\bn|),
\end{align}
where $\sigma$ is a \emph{2-permutation}, as defined in \S\ref{ss:FM_2^W_construction}.
The idea of this map is very simple.
The map $f_\sigma$ forgets the data of the lines, then labels the points according to the 2-permutation $\sigma$.
Then $f_\sigma$ extends continuously to the boundary of $W_\bn$; it is an embedding on the interior of its domain, but contracts some boundary cells.

\begin{example}
\label{ex:W_111}
In this figure, we depict $W_{111}$ and its image under an appropriate map $f_\sigma$.
More precisely, we depict their nets --- to ``assemble'' both CW complexes, one would cut them out, then glue together like-numbered edges.
As is evident, most of the 2-cells of $W_{111}$ are contracted by $f_\sigma$.
\begin{figure}[H]
\centering
\includegraphics[width=0.85\textwidth]{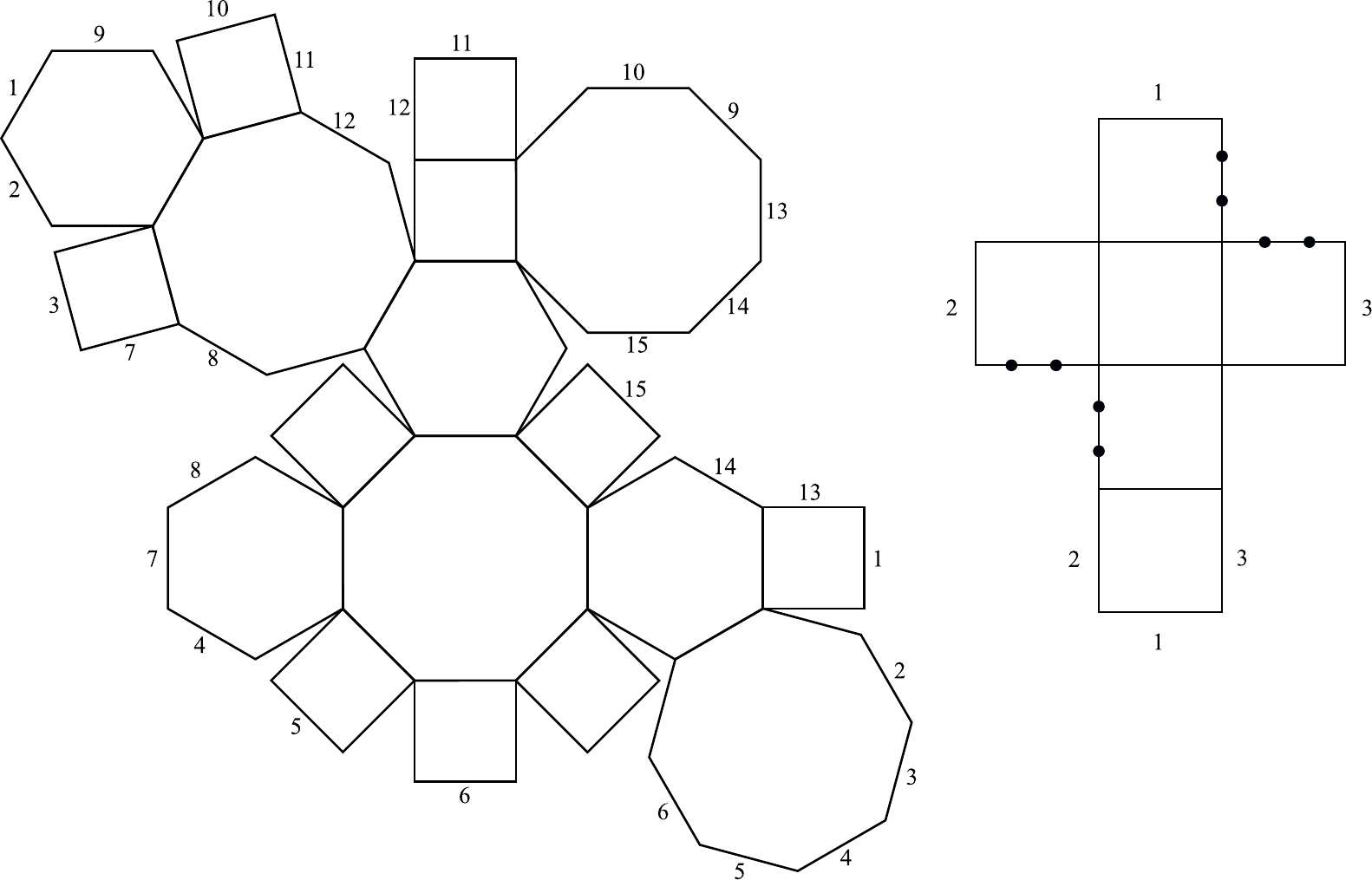}
\end{figure}
%\null\hfill$\triangle$
\end{example}

While it would take us too far afield to explain the relationship between $\FM_2$ and $\Symp$ (and their W-counterparts) in detail, let us indicate the basic idea.
$\Symp$, being an $(A_\infty,2)$-category, assigns to a chain in a 2-associahedron $W_\bn$ an operation on 2-morphisms.
(For instance: the objects of $\Symp$ are symplectic manifolds, and given two objects $M_0, M_1$, the 1-morphism category is $\Fuk(M_0^-\times M_1)$; 2-associahedra $W_n$, where $n$ is a single positive integer, act on this Fukaya category by the usual $A_\infty$-operations.)
The current definition of an $(A_\infty,2)$-category, appearing in \cite{bc}, does not equip identity 1-morphisms with all the possible structure.
Indeed, when defining operations on 2-morphisms in the situation where some of the 1-morphisms are identities, those 1-morphisms should be allowed to be ``moved past'' the other 1-morphisms.
To make this precise, one exactly needs to understand the maps $f_\sigma$, and to equip their targets with a CW structure so that $f_\sigma$ is cellular.
One way to proceed toward this goal is to first decompose $\FM_2^W$ so that $f_\sigma^W$ is cellular, and next construct coherent homeomorphisms $W_\bn \cong W_\bn^W$ and $\FM_2(k) \cong \FM_2^W(k)$.

The following result therefore shows the way toward a connection between the symplectic $(A_\infty,2)$-category and $\FM_2^W$.
It is an immediate consequence of our construction of $W_\bn^W$ and $\FM_2^W(k)$, and it forms the content of Remark~\ref{rem:FM_W} below.

\medskip

\noindent
{\bf Proposition.}
{\it Fix $r \geq 1$, $\bn \in \bZ_{\geq0}^r\setminus\{\bzero\}$, and a 2-permutation $\sigma$ of type $\bn$.
Then the associated map
\begin{align}
f_\sigma^W\colon W_\bn^W \to \FM_2^W(|\bn|)
\end{align}
is cellular.}

\subsection{Future directions}

The author plans to develop several aspects of the current paper.
In particular:

\begin{itemize}
%\item
%We do not address regularity in this paper.
%In future work with Paolo Salvatore, we plan to show that the CW complexes we construct here are regular.
%Moreover, we plan to show that our CW complexes are realizations of certain simplicial sets.
%
%\smallskip
%
\item
With several collaborators, the author plans to extend this work to produce cellular decompositions of $\FM_k^W$ for all $k \geq 1$, and to show that $\FM_k^W$ is isomorphic to $\FM_k$ in $\Top$.

\smallskip

\item
This paper can be construed as a way of incorporating identity 1-morphisms into the symplectic $(A_\infty,2)$-category.
The author plans to formalize this in future work on the algebra of $(A_\infty,2)$-categories.

\smallskip

\item
We plan to upgrade this work to give a cellular model for the framed analogue of the Fulton--MacPherson operad.
This suggests a way of endowing symplectic cohomology with a chain-level BV-algebra structure, which is the subject of Conjecture 2.6.1 from \cite{abouzaid}.
\end{itemize}

\subsection{Acknowledgments}
This paper is a solution to homework problem \#12 from Paul Seidel's course on Categorical Dynamics and Symplectic Topology at MIT in Spring 2013.
The author thanks Prof.\ Seidel for his patience.

Jacob Lurie drew an analogy that suggested to the author that there must be a link between $(A_\infty,2)$-categories and $E_2$-algebras.
Alexander Voronov explained to the author the colorful history surrounding this problem.
A conversation with Naruki Masuda, Hugh Thomas, and Bruno Vallette led the author to think about replacing $\FM_2$ with a ``W-construction version'' thereof.
The author thanks Dean Barber, Michael Batanin, Sheel Ganatra, Ezra Getzler, Mikhail Kapranov, Ben Knudsen, Paolo Salvatore, and Dev Sinha for their interest and encouragement.

The author was supported by an NSF Mathematical Sciences Postdoctoral Research Fellowship and by an NSF Standard Grant (DMS-1906220).
He thanks the Institute for Advanced Study, the Mathematical Sciences Research Institute, and the University of Southern California for providing excellent working conditions during the period when this work was carried out.

\section{A ``W-version'' of the 2-associahedra}
\label{s:W_n^W}

In this section, we construct a ``W-version'' of the 2-associahedra.
(The 2-associahedra were originally defined in \cite{b:2-associahedra}.)
This is an essential ingredient in our definition of $\FM_2^W(k)$, which will appear in \S\ref{s:FM_2^W}.

\subsection{A warm-up: $K^W$, i.e.\ $W(\Ass)$, i.e.\ a W-version of the associahedra}

In this subsection, we recall a certain operad, which we will denote by $K^W = \bigl(K_r^W\bigr)_{r \geq 1}$.
This is simply the Boardman--Vogt W-construction applied to the associative operad $\Ass$.
We construct only $K^W$ rather than recalling the general definition of the W-construction, because this one-off construction will be a useful warm-up to our construction of $W^W$ later in this section.
As noted in \cite{barber}, $K^W$ is isomorphic in $\Top$ to the associahedral operad $K$.

The following proposition summarizes what we will prove about $K^W$.

\begin{proposition}
\label{prop:K_r^W_operad}
The spaces $\bigl(K_r^W\bigr)_{r\geq 1}$ form a non-$\Sigma$ operad of CW complexes, and the composition maps
\begin{align}
\circ_i
\colon
K_r^W \times K_s^W \to K_{r+s-1}^W
\end{align}
defined in Def.~\ref{def:K_r_composition} are cellular.
\end{proposition}

\noindent
We will prove Prop.~\ref{prop:K_r^W_operad} at the end of the current subsection.

We begin with a definition of rooted ribbon trees.
Stable rooted ribbon trees with $r$ leaves index the strata of the associahedron $K_r$, and they will be an integral part of the definition of $K_r^W$.

\begin{definition}[Def.~2.2, \cite{b:2-associahedra}]
\label{def:Krtree_set}
A \emph{rooted ribbon tree} (RRT) is a tree $T$ with a choice of a root $\alpha_\root \in T$ and a cyclic ordering of the edges incident to each vertex; we orient such a tree toward the root.
We say that a vertex $\alpha$ of an RRT $T$ is \emph{interior} if the set $\incom(\alpha)$ of its incoming neighbors is nonempty, and we denote the set of interior vertices of $T$ by $T_\inte$.
An RRT $T$ is \emph{stable} if every interior vertex has at least 2 incoming edges.
We define $K^\tree_r$ to be the set of all isomorphism classes of stable rooted ribbon trees with $r$ leaves.

We denote the $i$-th leaf of an RRT $T$ by $\lambda_i^T$.
For any $\alpha, \beta \in T$, $T_{\alpha\beta}$\label{p:Talphabeta} denotes those vertices $\gamma$ such that the path $[\alpha,\gamma]$\label{p:path} from $\alpha$ to $\gamma$ passes through $\beta$.
We denote $T_\alpha \coloneqq T_{\alpha_\root\alpha}$.
%\null\hfill$\triangle$
\end{definition}

\begin{remark}
Ribbon trees (respectively rooted ribbon trees) are often referred to as planar trees (respectively planted trees).
%\null\hfill$\triangle$
\end{remark}

Next, we define a version of RRTs with internal edge lengths.

\begin{definition}
\label{def:metric_RRT}
A \emph{metric RRT} $\bigl(T,(\ell_e)\bigr)$ is the following data:
\begin{itemize}
\item
An RRT $T$.

\item
For every edge $e$ of $T$ not incident to a leaf (but possibly incident to the root), a length $\ell_e \in [0,1]$.
\end{itemize}
We call this a \emph{metric RRT of type $T$}.
%\null\hfill$\triangle$
\end{definition}

Now we will define a ``dimension'' function $d$ on stable RRTs.
%Stable RRTs index the strata of $K_r$; $d$ assigns to a stable RRT the dimension of the corresponding stratum of $K_r$.

\begin{definition}[Definition 2.4, \cite{b:2-associahedra}]
\label{def:RRT_dim}
For $T$ a stable RRT in $K_r^\tree$, we define its \emph{dimension} $d(T) \in [0,r-2]$ like so:
\begin{align}
\label{eq:RRT_dim}
d(T) \coloneqq r - \#\!T_\inte - 1.
\end{align}
%\null\hfill$\triangle$
\end{definition}

\begin{definition}
\label{def:C_T}
Given a stable tree $T$, the \emph{cell associated to $T$} is denoted by $C_T$ and is defined to consist of all metric RRTs of type $T$.
%\null\hfill$\triangle$
\end{definition}

\noindent
Note that we can canonically identify $C_T$ with the closed cube of dimension equal to the number of internal edges of $T$.
That is:
\begin{align}
C_T
\cong
[0,1]^{\#T_\inte - 1}
=
[0,1]^{r - 2 - d(T)}.
\end{align}
As we will see, $K_r^W$ is $(r-2)$-dimensional; it follows that $d(T)$ is the codimension of $C_T$ in $K_r^W$.
(The unfortunate clash of terminology between ``dimension'' and ``codimension'' is due to the fact that in $K_r$, the cell indexed by $T$ has dimension $d(T)$.)

We now define $K_r^W$ by taking the union of the cells $C_T$ for $T$ any stable RRT with $r$ leaves, then collapsing edges of length 0.

\begin{definition}
\label{def:K_r^W}
Given $r \geq 1$, we define $K_r^W$ to be the following quotient:
\begin{align}
K_r^W
\coloneqq
\Bigl(\:\bigsqcup_{T \in K_r^\tree} C_T\biggr)\Big/\sim.
\end{align}
Here $\sim$ identifies $\bigl(T,(\ell_e)\bigr)$ and $\bigl(T',(\ell'_e)\bigr)$ if, after collapsing all edges $e$ of $T$ with $\ell_e=0$ and all edges $e$ of $T'$ with $\ell_e' = 0$, both metric RRTs reduce to the same metric RRT $\bigl(T'',(\ell_e'')\bigr)$.
%\null\hfill$\triangle$
\end{definition}

\begin{example}
In the following figure, we depict the CW complex $K_4^W$.
\begin{figure}[H]
\centering
\def\svgwidth{0.55\columnwidth}
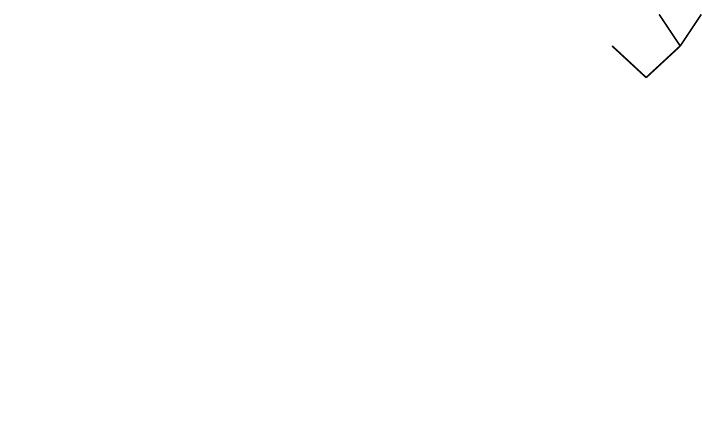
\end{figure}
\noindent
Note that this is a refinement of $K_4$, which (as a CW complex) is a pentagon.
We have labeled the open top cells by the metric stable RRTs that they parametrize, where each $a$ and $b$ is allowed to vary in $[0,1]$.
The closed top cells are glued together along the cells where some of the edge lengths are 0 --- for instance, we have indicated how the top and top-right cubes are joined along the internal edge of the pentagon where the edge length $b$ in both cells becomes 0.
The boundary of $K_r^W$ is the union of cells where at least one edge length is 1.
%\null\hfill$\triangle$
\end{example}

Finally, we define a simplicial refinement of the CW structure on $K_r^W$.
To approach this, we note that if $P$ is the poset $\{0,1\}^k$, where $\sigma_1 < \sigma_2$ if $\sigma_2$ can be gotten by changing some of the 0's of $\sigma_1$ to 1's, then the nerve of $P$ is a simplicial decomposition of the cube $[0,1]^k$.
More concretely, the top simplices are the sets of the form
\begin{align}
\bigl\{(x_1,\ldots,x_k) \in [0,1]^k
\:|\:
0 < x_{\sigma(1)} < \cdots < x_{\sigma(k)} < 1\bigr\},
\end{align}
where $\sigma$ is a permutation on $k$ letters.
The remaining simplices are the result of replacing some of these inequalities by equalities.

%Given a stable tree $T \in K_r$, we can identify its associated cell $C_T$ with the cube $[0,1]^{r-2-d(T)}$.

\begin{definition}
We refine the CW structure on $K_r^W$ by decomposing each cell $C_T$ in $K_r^W$ like so: we make the identification $C_T \cong [0,1]^{r-2-d(T)}$, then perform the simplicial decomposition described in the previous paragraph.
This refinement equips $K_r^W$ with a simplicial decomposition.
%\null\hfill$\triangle$
\end{definition}

\begin{example}
In the following figure, we depict the simplicial complex $K_4^W$.
\begin{figure}[H]
\centering
\def\svgwidth{0.5\columnwidth}
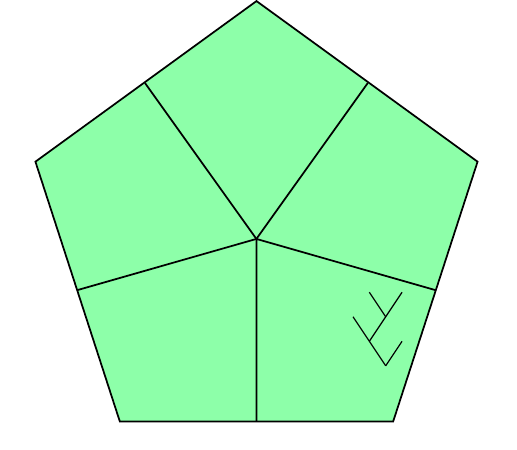
\end{figure}
\noindent
This is the refinement of our initial, cubical, CW decomposition of $K_r^W$ gotten by subdividing each of the five squares into two triangles.
We indicate the new edges by coloring them blue.
%\null\hfill$\triangle$
\end{example}

Now that we have constructed the spaces $K_r^W$, we can prove Prop.\ 2.1, which states that $\bigl(K_r^W\bigr)$ is a non-$\Sigma$ operad and that the operad maps are cellular.

\begin{definition}
\label{def:K_r_composition}
Fix $r$, $s$, and $i \in [1,r]$.
We wish to define the composition map
\begin{align}
\circ_i
\colon
K_r^W \times K_s^W \to K_{r+s-1}^W.
\end{align}
We do so cell by cell.
That is, fix cells $C_T \subset K_r^W$ and $C_{T'} \subset K_s^W$.
Define $T''$ to be the result of grafting $T'$ to the $i$-th leaf of $T$.
Then we define $\circ_i$ on $C_T \times C_{T'}$ like so: given collections of edge lengths on $T$ and $T'$, combine them to produce a collection of edge lengths on $T''$, where we assign to the single newly-formed interior edge the length 1.
%\null\hfill$\triangle$
\end{definition}

\begin{proof}[Proof of Prop.~\ref{prop:K_r^W_operad}]
Fix $r$, $s$, and $i \in [1,r]$, and consider the composition map
\begin{align}
\circ_i
\colon
K_r^W \times K_s^W \to K_{r+s-1}^W.
\end{align}
To show that $\circ_i$ is cellular, let's consider the restriction of $\circ_i$ to a product $C_T \times C_{T'}$ of closed cubes, for $T \in K_r$ and $T' \in K_s$.
Denote by $T''$ the tree obtained by grafting the root of $T'$ to the $i$-th leaf of $T$.
Then $\circ_i$ includes $C_T \times C_{T'}$ into $C_{T''}$ as the face gotten by requiring the outgoing edge of the root of $T'$ to have length 1.
The CW structure of this face of $C_{T''}$ is finer than that of $C_T \times C_{T''}$, so $\circ_i$ is indeed cellular. 
\end{proof}

\subsection{Metric tree-pairs and the definition of $W_\bn^W$}

Just as we defined $K_r^W$ to be the parameter space of metric stable RRTs, we will define $W_\bn^W$ to parametrize metric stable tree-pairs.
The definition of metric stable tree-pairs is somewhat involved, so we devote the current subsection to this definition.

Before defining metric stable tree-pairs, we recall the definition of stable tree-pairs.

\begin{definition}[Def.~3.1, \cite{b:2-associahedra}]
\label{def:2-associahedra}
A \emph{stable tree-pair of type $\bn$} is a datum $2T = T_b \sr{f}{\to} T_s$, with $T_b, T_s, f$ described below:
\begin{itemize}
\item
The \emph{bubble tree} $T_b$ is an RRT whose edges are either solid or dashed, which must satisfy these properties:
\begin{itemize}
\item
The vertices of $T_b$ are partitioned as $V(T_b) = V_\comp \sqcup V_\seam \sqcup V_\mk$, where:
\begin{itemize}
\item
every $\alpha \in V_\comp$ has $\geq 1$ solid incoming edge, no dashed incoming edges, and either a dashed or no outgoing edge;

\item
every $\alpha \in V_\seam$ has $\geq 0$ dashed incoming edges, no solid incoming edges, and a solid outgoing edge; and

\item
every $\alpha \in V_\mk$ has no incoming edges and either a dashed or no outgoing edge.
\end{itemize}
We partition $V_\comp \eqqcolon V_\comp^1 \sqcup V_\comp^{\geq2}$ according to the number of incoming edges of a given vertex.

\item
({\sc stability})
If $\alpha$ is a vertex in $V_\comp^1$ and $\beta$ is its incoming neighbor, then $\#\!\incom(\beta) \geq 2$; if $\alpha$ is a vertex in $V_\comp^{\geq2}$ and $\beta_1,\ldots,\beta_\ell$ are its incoming neighbors, then there exists $j$ with $\#\!\incom(\beta_j) \geq 1$.
\end{itemize}
	
\item
The \emph{seam tree} $T_s$ is an element of $K_r^\tree$.

\item
The \emph{coherence map} is a map $f\colon T_b \to T_s$ of sets having these properties:
\begin{itemize}
\item
$f$ sends root to root, and if $\beta \in \incom(\alpha)$ in $T_b$, then either $f(\beta) \in \incom(f(\alpha))$ or $f(\alpha) = f(\beta)$.

\item
$f$ contracts all dashed edges, and every solid edge whose terminal vertex is in $V_\comp^1$.

\item
For any $\alpha \in V_\comp^{\geq2}$, $f$ maps the incoming edges of $\alpha$ bijectively onto the incoming edges of $f(\alpha)$, compatibly with $<_\alpha$ and $<_{f(\alpha)}$.

\item
$f$ sends every element of $V_\mk$ to a leaf of $T_s$, and if $\lambda_i^{T_s}$ is the $i$-th leaf of $T_s$, then $f^{-1}\{\lambda_i^{T_s}\}$ contains $n_i$ elements of $V_\mk$, which we denote by $\mu_{i1}^{T_b},\ldots,\mu_{in_i}^{T_b}$.
\end{itemize}
\end{itemize}
We denote by $W_\bn^\tree$\label{p:Wntree} the set of isomorphism classes of stable tree-pairs of type $\bn$.
%new
Here an isomorphism from $T_b \sr{f}{\to} T_s$ to $T_b' \sr{f'}{\to} T_s'$ is a pair of maps $\varphi_b\colon T_b \to T_b'$ and $\varphi_s\colon T_s \to T_s'$ that fit into a commutative square in the obvious way and that respect all the structure of the bubble trees and seam trees.
%\null\hfill$\triangle$
\end{definition}

Next, we define metric stable tree-pairs.
This notion is more subtle than that of metric stable RRTs, because we must impose conditions on the edge-lengths.
(This should be compared with \cite[\S3]{bo}, where Bottman--Oblomkov imposed similar constraints in order to define local charts on a complexified version of $W_\bn$.)

\begin{definition}
\label{def:metric_stable_tree-pair}
A \emph{metric stable tree-pair} $\bigl(2T, (L_e), (\ell_e)\bigr)$ is the following data:
\begin{itemize}
\item
A stable tree-pair $2T$.

\item
For every interior dashed edge $e$ of $T_b$, a length $L_e \in [0,1]$, and for every interior edge $e$ of $T_s$, a length $\ell_e \in [0,1]$, subject to the following coherence conditions (where for convenience, we set $L_\alpha \coloneqq L_e$ for $\alpha \in V_\comp(T_b)\setminus\{\alpha_\root\}$ and $e$ the outgoing edge of $\alpha$, and similarly for the edge-lengths in $T_s$):
\begin{itemize}
\item
For every $\alpha_1,\alpha_2 \in V_\comp^{\geq2}(T_b)$ and $\beta \in V_\comp^1(T_b)$ with $f(\alpha_1) = f(\alpha_2) = f(\beta)$, we require:
\begin{align}
\label{eq:coherences_1}
\max_{\gamma \in [\alpha_1,\beta)} L_\gamma
=
\max_{\gamma \in [\alpha_2,\beta)} L_\gamma.
\end{align}

\item
For every $\rho \in V_\inte(T_s)\setminus\{\rho_\root\}$ and $\alpha \in V_\comp^{\geq2}(T_b)$ with $f(\alpha) = \rho$, we require:
\begin{align}
\label{eq:coherences_2}
\ell_\rho
=
\max_{\gamma \in [\alpha,\beta_\alpha)} L_\gamma,
\end{align}
where we define $\beta_\alpha$ to be the first element of $V_\comp^{\geq2}(T_b)$ that the path from $\alpha$ to $\alpha_\root$ passes through.
\end{itemize}
\end{itemize}
%\null\hfill$\triangle$
\end{definition}

Finally, we recall the \emph{dimension} of a stable tree-pair.
Similarly to the dimension of a stable RRT, this will be the codimension in $W_\bn^W$ of the cell corresponding to the stable tree-pair in question.

\begin{definition}[Definition 3.3, \cite{b:2-associahedra}]
\label{def:tree-pair_dim}
For $2T$ a stable tree-pair, we define the \emph{dimension} $d(2T) \in [0,|\bn|+r-3]$ like so:
\begin{align} \label{eq:tree-pair_dim}
d(2T) \coloneqq |\bn| + r - \#\!V^1_\comp(T_b) - \#\!(T_s)_\inte - 2.
\end{align}
%\null\hfill$\triangle$
\end{definition}

We are now prepared to define $W_\bn^W$, the ``W-version'' of the 2-associahedron.
We will define $W_\bn^W$ by attaching together the cells $C_{2T}$, which consist of metric stable tree-pairs. 

\begin{definition}
\label{def:C_2T}
Given a stable tree-pair $2T$, the \emph{cell associated to $2T$} is the collection of all metric stable tree-pairs of type $2T$.
We denote this cell by $C_{2T}$.
%\null\hfill$\triangle$
\end{definition}

\noindent
Note that we can identify $C_{2T}$ with the subset of the cube $[0,1]^k$ defined by the equalities \eqref{eq:coherences_1} and \eqref{eq:coherences_2}, where $k$ is the number of interior dashed edges of $T_b$ plus the number of interior edges of $T_s$.

\begin{definition}
\label{def:W_n^W}
Fix $r \geq 1$ and $\bn \in \bZ^r_{\geq0} \setminus \{\bzero\}$.
We define $W_\bn^W$ similarly to how we defined $K_r^W$ in Def.~\ref{def:K_r^W}:
\begin{align*}
W_\bn^W
\coloneqq
\biggl(\:\bigsqcup_{2T \in W_\bn^\tree} C_{2T}\biggr)\Big/\sim.
\end{align*}
The quotient here is somewhat subtler than the quotient that appeared in Def.~\ref{def:K_r^W}, specifically when it comes to $T_b$.
In $T_s$, we simply contract any edges of length 0.
We indicate in the following figure how to perform the necessary contractions in $T_b$ when some edge-lengths are 0:
\begin{figure}[H]
\centering
\def\svgwidth{0.9\columnwidth}
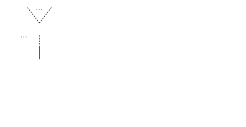
\end{figure}
\noindent
The reader should think of the left contraction as undoing a type-1 move (as in \cite[\S3.1]{b:2-associahedra}), whereas the right contraction undoes either a type-2 or a type-3 move.
Note that we are using the coherences enforced in Def.~\ref{def:metric_stable_tree-pair} --- for instance, these mean that we do not have to consider a situation as in the right-hand side of the above figure, but where only some of the edge-lengths in this portion of $T_b$ are 0.
%\null\hfill$\triangle$
\end{definition}

\begin{example}
In the following figure, we depict the CW complex $W_{21}^W$.
\begin{figure}[H]
\centering
\def\svgwidth{1.0\columnwidth}
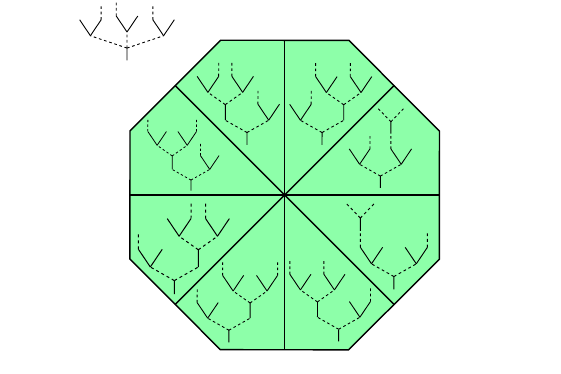
\end{figure}
\noindent
Each of the parameters $a$ and $b$ lie in $[0,1]$; they do not have the same meaning across different cells.
The eight interior edges (resp.\ sixteen boundary edges) correspond to the loci in the top cells where a parameter goes to 0 (resp.\ to 1).
%\null\hfill$\triangle$
\end{example}

Finally, we refine the CW structure on $W_\bn^W$ to a simplicial decomposition.

\begin{lemma}
\label{lem:W^W_simplicial}
Fix a stable tree-pair $2T$.
For every simplex $S$ in the standard simplicial decomposition of $[0,1]^k \supset C_{2T}$, $S$ is either contained in $C_{2T}$ or disjoint from it.
The collection of such simplices that are contained in $C_{2T}$ form a simplicial decomposition of $C_{2T}$.
%\null\hfill$\triangle$
\end{lemma}

\begin{proof}
Fix a simplex $S$.
$S$ is defined by a collection of equalities and inequalities of the form
\begin{align}
\label{eq:simplex_conditions}
0 \:*\: x_{\sigma(1)} \:*\: \cdots \:*\: x_{\sigma(k)} \:*\: 1,
\end{align}
where each ``$*$'' is either a ``$<$'' or an ``$=$'' and where $\sigma$ is a permutation on $k$ letters.
After imposing these (in)equalities, the left- and right-hand sides of the equalities \eqref{eq:coherences_1} and \eqref{eq:coherences_2} become single variables.
This collection of equalities will either be always satisfied or never satisfied, depending on the constraints in \eqref{eq:simplex_conditions}.
Depending on which of these is the case, $S$ is either contained in $C_{2T}$ or disjoint from it.

It follows immediately that the collection of simplices that are contained in $C_{2T}$ form a simplicial decomposition of $C_{2T}$.
\end{proof}

\begin{example}
In the following figure, we illustrate the closed cell in $W_{40}^W$ associated to the underlying tree-pair of the (top-dimensional) metric tree-pair shown on the right:
\begin{figure}[H]
\centering
\def\svgwidth{1.0\columnwidth}
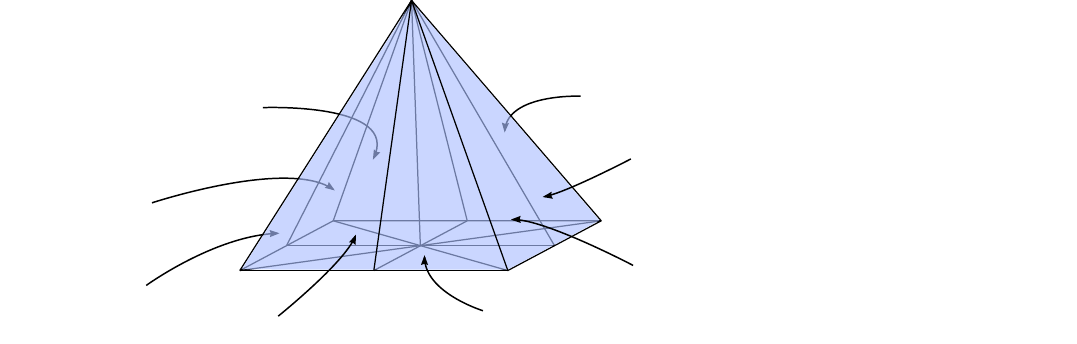
\end{figure}
\noindent
The restriction on the lengths $a,b,c,d \in [0,1]$ is that they must satisfy $\max(a,b) = \max(c,d)$; as a result, this cell has the CW type of a square pyramid.

We indicate the simplicial refinement of this cell: the square pyramid is subdivided into eight 3-simplices, which are defined by imposing inequalities and equalities as shown in this figure.
%\null\hfill$\triangle$
\end{example}

\section{The construction of $\FM_2^W$}
\label{s:FM_2^W}

In this section, the d\'enouement of this paper, we will construct a collection of CW complexes $\bigl(\FM_2^W(k)\bigr)_{k\geq1}$ and a collection of operations
\begin{align}
\circ_i
\colon
\FM_2^W(k) \times \FM_2^W(\ell)
\to
\FM_2^W(k+\ell-1),
\end{align}
such that these data form an operad.

We will now give an overview of our construction of $\FM_2^W(k)$.
This is an expansion of Step 2 in the overview we gave in \S\ref{ss:overview}, and we label the parts accordingly:

\bigskip

\noindent
{\bf 2a.}
Each open Getzler--Jones cell in $\FM_2(k)$ can be identified with a product of open 2-associahedra, i.e.\ a product of the form $\mr W_{\bm^1} \times \cdots \times \mr W_{\bm^a}$ (where ``$\mr X$'' is our notation for the interior of a space $X$).
For each such open cell, we replace these 2-associahedra by their W-construction equivalents, thusly: $\mr W_{\bm^1}^W \times \cdots \times \mr W_{\bm^a}^W$.
This product comes with the product CW structure, and we refine this in a way that endows $\mr W_{\bm^1}^W \times \cdots \times \mr W_{\bm^a}^W$ with the structure of a simplicial complex.

\medskip

\noindent
{\bf 2b.}
While an open Getzler--Jones cell can be identified with a product $\mr W_{\bm^1} \times \cdots \times \mr W_{\bm^a}$ of 2-associahedra, their compactifications (in $\FM_2(k)$ and $W_{\bm^1}\times\cdots\times W_{\bm^a}$, respectively) are different: the compactification of the former is smaller than the compactification of the latter.
This is reflected in how we glue our products $\mr W_{\bm^1}^W \times \cdots \times \mr W_{\bm^a}^W$ together.
Specifically, we perform this gluing by applying a quotient map to each simplex in the boundary of $W_{\bm^1}^W \times \cdots \times W_{\bm^a}^W$.
This quotient map is closely related to the maps $f_\sigma\colon W_\bn \to \FM_2(k)$ that we described in \S\ref{ss:relationship}: they reflect the fact that the compactification used to define $W_\bn$ allows lines with no marked points, whereas the compactification of a Getzler--Jones cell does not allow this.

\bigskip

The following is the main result of this section, which we stated in the introduction and record again here:

\medskip

\noindent
{\bf Main theorem.}
{\it The spaces $\bigl(\FM_2^W(k)\bigr)_{k\geq1}$ together with the composition operations $\circ_i$ defined in Def.~\ref{def:FM_2^W_composition} form a non-$\Sigma$ operad, and the composition maps
\begin{align}
\circ_i
\colon
\FM_2^W(k) \times \FM_2^W(\ell)
\to
\FM_2^W(k+\ell-1)
\end{align}
are cellular.}

\bigskip

\begin{proof}
Combine Lemmata \ref{lem:it's_an_operad} and \ref{lem:compositions_are_cellular} below.
\end{proof}

\subsection{Quotient maps on 2-associahedra}

Before we can define the quotient involved in \eqref{eq:FM_2^W_definition}, we will define for every cell $F$ in $\partial W_\bn^W$ a map $q_F$ from $F$ to a certain product of 2-associahedra, where this target will vary for difference choices of $F$.
We begin with two preliminary definitions.

\begin{definition}
\label{def:pi^tree}
Fix $r \geq 1$ and $\bn \in \bZ_{\geq0}^r\setminus\{\bzero\}$, and fix $i \in [1,r]$ such that $n_i = 0$.
Define $\wt\bn \coloneqq (n_1,\ldots,n_{i-1},n_{i+1},\ldots,n_r)$.
We then define a map of posets $\pi_i^\tree\colon W_\bn^\tree \to W_{\wt\bn}^\tree$ by applying the following procedure to $2T = T_b \sr{f}{\to} T_s \in W_\bn^\tree$:
\begin{itemize}
\item[1.]
Denote by $e_0$ the edge in $T_s$ incident to the $i$-th leaf $\lambda_i^{T_s}$.
If $e$ is a solid edge in $T_b$ that is mapped identically under $f$ to $e_0$, then we delete $e$.
Next, we delete $e_0$.
We modify $f$ in the obvious way.

\medskip

\item[2.]
After performing these deletions, our tree-pair may no longer be stable.
We rectify this in $T_b$ resp.\ $T_s$ by performing the contractions indicated on the left resp.\ right:
\begin{figure}[H]
\centering
\def\svgwidth{0.4\columnwidth}
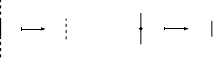
\end{figure}
\noindent
More specifically, we perform these contractions as many times as necessary for the tree-pair to be stable.
\end{itemize}
\noindent
Denoting the end result of this procedure by $\wt{2T}$, we define $\pi_i^\tree(2T) \coloneqq \wt{2T}$.

Next, we define another map of posets.
Fix $r \geq 1$ and $\bn \in \bZ_{\geq0}^r\setminus\{\bzero\}$.
Denote by $\wt\bn$ the result of deleting all the zeroes from $\bn$, and set $\wt r$ to be the length of $\wt\bn$.
We define $\pi^\tree\colon W_\bn^\tree \to W_{\wt\bn}^\tree$ by applying the map $\pi_i^\tree$ once for each $i$ with $n_i = 0$.
%\null\hfill$\triangle$
\end{definition}

\noindent
It is not hard to check that the choices implicit in this definition do not matter, and that the resulting maps are indeed maps of posets.

\begin{definition}
\label{def:pi^W}
Fix $r \geq 1$ and $\bn \in \bZ^r_{\geq0}\setminus\{\bzero\}$.
We define a map $\pi^W\colon W_\bn^W \to W_{\wt\bn}^W$ in the same fashion as $\pi^\tree$, with the provision that when we contract adjacent edges of lengths $\ell_1$ and $\ell_2$ (whether in $T_b$ or $T_s$), we equip the resulting edge with length $\max(\ell_1,\ell_2)$.
\end{definition}

Next, we recall a W-version analogue of two properties of the 2-associahedra.

\medskip

\noindent
{\bf W-version analogue of} \textsc{(forgetful)} {\bf property of Theorem 4.1, \cite{b:2-associahedra}.}
Fix $r \geq 1$ and $\bn \in \bZ^r_{\geq0}\setminus\{\bzero\}$.
There is a surjection $W_\bn^W \to K_r^W$, which sends a metric stable tree-pair $\bigl(T_b \sr{f}{\mapsto} T_s, (L_e), (\ell_e)\bigr)$ to the metric stable RRT $\bigl(T_s, (\ell_e)\bigr)$.
%\null\hfill$\triangle$

\medskip

\noindent
{\bf W-version analogue of} \textsc{(recursive)} {\bf property of Theorem 4.1, \cite{b:2-associahedra}.}
Fix a stable tree-pair $2T = T_b \sr{f}{\to} T_s \in W_\bn^\tree$.
There is an inclusion of CW complexes
\begin{align}
\Gamma_{2T} \colon \prod_{
{\alpha \in V_\comp^1(T_b),}
\atop
{\incom(\alpha)=(\beta)}
} W_{\#\!\incom(\beta)}^W
\times
\prod_{\rho \in V_\inte(T_s)} \prod^{K_{\#\!\incom(\rho)}^W}_{
{\alpha\in V_\comp^{\geq2}(T_b)\cap f^{-1}\{\rho\},}
\atop
{\incom(\alpha)=(\beta_1,\ldots,\beta_{\#\!\incom(\rho)})}
}
\hspace{-0.25in} W^W_{\#\!\incom(\beta_1),\ldots,\#\!\incom(\beta_{\#\!\incom(\alpha)})}
\hookrightarrow W_\bn^W,
\end{align}
where the superscript on one of the product symbols indicates that it is a fiber product with respect to the maps described in \textsc{(forgetful)}.

The map $\Gamma_{2T}$ defined in \cite{b:2-associahedra}, which is defined for the posets $W_\bn^\tree$, is defined by attaching stable tree-pairs together in a way specified by the stable tree-pair $2T$.
This map is similar, but we are attaching together \emph{metric} stable tree-pairs.
We assign the length 1 to the edges along which we attach the trees.
(The image of $\Gamma_{2T}$ is a union of cells in $\partial W_\bn^W$.)
%\null\hfill$\triangle$

\medskip

We can now define the quotient maps $q_F$ on $W_\bn^W$.

\begin{definition}
\label{p:q_F}
Fix $r \geq 1$, $\bn \in \bZ_{\geq0}^r\setminus\{\bzero\}$, a stable type-$\bn$ tree-pair $\wt{2T}$, and a face $F$ of the associated cell $C_{\wt{2T}}$ in $W_\bn^W$ with the property that $F$ lies in $\partial W_\bn^W$.
(Equivalently, the metric tree-pairs in $F$ have at least one length that is identically equal to 1.)
The \emph{quotient map associated to $F$} is a map $q_F$ from $F$ to a product of 2-associahedra.
Given a metric stable tree-pair $\bigl(2T,(L_e),(\ell_e)\bigr)$, we define its image under $\pi$ in the following fashion:
\begin{itemize}
\item[1.]
Break up $T_b$ and $T_s$ along the edges that are identically 1 in $F$.
Equivalently, choose $2T$ of minimal dimension with the property that $F$ lies in the image of $\Gamma_{2T}$, then identify $F$ as a top cell in a product of fiber products of the following form:
\begin{align}
\prod_{
{\alpha \in V_\comp^1(T_b),}
\atop
{\incom(\alpha)=(\beta)}
} W_{\#\!\incom(\beta)}^W
\times
\prod_{\rho \in V_\inte(T_s)} \prod^{K_{\#\!\incom(\rho)}^W}_{
{\alpha\in V_\comp^{\geq2}(T_b)\cap f^{-1}\{\rho\},}
\atop
{\incom(\alpha)=(\beta_1,\ldots,\beta_{\#\!\incom(\rho)})}
}
\hspace{-0.25in} W^W_{\#\!\incom(\beta_1),\ldots,\#\!\incom(\beta_{\#\!\incom(\alpha)})}.
\end{align}
As a result, we obtain a list of metric stable tree-pairs, which we can regard as lying inside a product $W_{\bm^1}^W \times \cdots \times W_{\bm^a}^W$.

\medskip

\item[2.]
We then apply the map $\pi^W$ to each of the factors in the product just recorded, hence producing an element of $W_{\wt{\bm^1}}^W \times \cdots \times W_{\wt{\bm^a}}^W$.
(As in Defs.~\ref{def:pi^tree} and \ref{def:pi^W}, $\wt{\bm^i}$ denotes the result of removing the 0's from $\bm^i$.)
\end{itemize}
%\null\hfill$\triangle$
\end{definition}

\noindent
Note that for two cells $F_1, F_2$ in the boundary of $W_\bn^W$, the targets of $q_{F_1}$ and $q_{F_2}$ are typically different.

\begin{example}
In the following figure, we illustrate several things about $W_{21}^W$:
\begin{figure}[H]
\centering
\def\svgwidth{0.9\columnwidth}
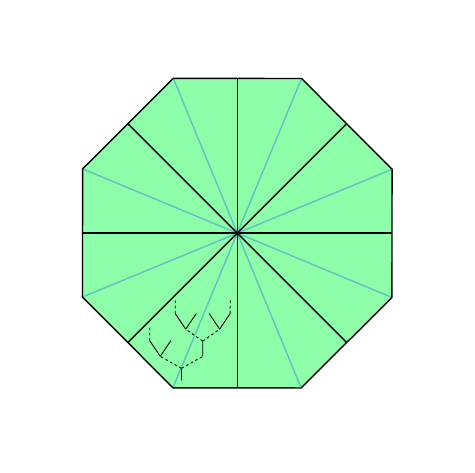
\end{figure}
\noindent
Initially, $W_{21}^W$ is an octagon, decomposed into eight squares; this is indicated by the black lines.
The simplicial refinement divides each square into two 2-simplices.
We have indicated the metric tree-pairs that correspond to each of the eight squares, as well as those corresponding to the sixteen 1-simplices that comprise $\partial W_{21}^W$.
(Some dashed edges are not labeled; these should be interpreted as having length $\max(a,b)$.)

Finally, we have indicated the behavior of the quotient maps on $W_{21}^W$.
This map are the identity on every edge except for those indicated in red.
Each pair of red edges is contracted to a point.
One reflection of this is that in Ex.~\ref{ex:W_111}, the octagons in $W_{111}$ are taken to the (cellular) hexagons in the Getzler--Jones cell indicted on the right.
%\null\hfill$\triangle$
\end{example}

\subsection{The construction of $\FM_2^W(k)$}
\label{ss:FM_2^W_construction}

In this subsection, we tackle the construction of $\FM_2^W(k)$.
First, we will describe our version of the Getzler--Jones cells.
Next, we will explain how to glue these spaces together.

To define the Getzler--Jones cells, we must introduce \emph{2-permutations}, which will allow us to enforce the alignment and ordering of special points on screens as in Fig.\ \ref{fig:GJ}.

\begin{definition}
\label{def:2-permutation}
Fix a finite set $A$.
A \emph{2-permutation $\sigma$ on $A$} is the following data:
\begin{itemize}
\item
An ordered decomposition
\begin{align}
A = A_1 \sqcup \cdots \sqcup A_r,
\end{align}
where $A_r$ is allowed to be empty.

\medskip

\item
For each $i$, a linear order on $A_i$.
\end{itemize}
We define the \emph{type} of $\sigma$ to be the vector $\bn \coloneqq \bigl(|A_1|,\ldots,|A_r|\bigr)$.
If $\sigma$ is a 2-permutation whose type $\bn$ has no zero entries, then we say that $\sigma$ \emph{has no empty part}.
%\null\hfill$\triangle$
\end{definition}

\begin{remark}
Note that a type-$\underbrace{(1,\ldots,1)}_r$ 2-permutation is exactly the data of a permutation on $r$ letters.
The same is true of a type-$(n)$ 2-permutation.
%\null\hfill$\triangle$
\end{remark}

\noindent
Next, we define a \emph{Getzler--Jones datum}, the set of which indexes the Getzler--Jones cells in $\FM_2^W(k)$.

\begin{definition}
\label{def:GJ_datum}
Fix $k \geq 2$.
A \emph{Getzler--Jones datum} consists of the following data:
\begin{itemize}
\item
A stable rooted tree $T$ with $k$ leaves, together with a numbering of its leaves from $1$ through $k$.

\medskip

\item
For every interior vertex $v \in T_\inte$, a 2-permutation $\sigma$ on its incoming vertices $V_\incom(T)$ such that $\sigma$ has no empty part.
\end{itemize}
We denote the type of the 2-permutation associated to $v$ by $\bn(v)$.
We will abuse notation and denote the entire Getzler--Jones datum by $T$.
%\null\hfill$\triangle$
\end{definition}

\noindent
Finally, we can define the \emph{Getzler--Jones cells of type $k$}.

\begin{definition}
\label{p:GJ_T}
Fix $k \geq 2$ and a Getzler--Jones datum $T$.
Then we make the following two definitions: 
\begin{align}
\GJ_T
\coloneqq
\prod_{v \in T_\inte} \mr W_{\bn(v)}^W,
\qquad
\wt\GJ_T
\coloneqq
\prod_{v \in T_\inte} W_{\bn(v)}^W.
\end{align}
We call $\GJ_T$ the \emph{Getzler--Jones cell $\GJ_T$ associated to $T$}, and we refer to $\GJ_T$ as a \emph{type-$k$ Getzler--Jones cell}.

In Lem.~\ref{lem:W^W_simplicial} we equipped $W_\bn^W$ with the structure of a simplicial complex, which induces a CW structure on $\GJ_T$ and $\wt\GJ_T$.
We refine these to equip $\GJ_T$ and $\wt\GJ_T$ with simplicial decompositions, in the fashion of Lem.~\ref{lem:W^W_simplicial}.
%\null\hfill$\triangle$
\end{definition}

\begin{remark}
The reason why we do not refer to $\wt\GJ_T$ as a ``closed Getzler--Jones cell'' is because it is \emph{not} the closure in $\FM_2^W(k)$ of $\GJ_T$.
In fact, it is larger than this closure.
Our reason for making this second definition is that $\wt\GJ_T$ will be an integral part of our definition of $\FM_2^W(k)$.
%\null\hfill$\triangle$
\end{remark}

We will define $\FM_2^W(k)$ as a quotient of the following form, where $T$ varies over type-$k$ Getzler--Jones data:
\begin{align}
\label{eq:FM_2^W_definition}
\FM_2^W(k)
\coloneqq
\Bigl(\coprod_T \: \wt\GJ_T\Bigr)\big/\sim.
\end{align}
The remaining ingredient is the collection of maps that we will use to attach these spaces.
As a consequence of the definition of these maps, $\FM_2^W(k)$ will decompose as a set into the union of all type-$k$ Getzler--Jones cells.

Finally, we come to the definition of $\FM_2^W(k)$.

\begin{definition}
Fix $k \geq 2$.
We construct $\FM_2^W(k)$ like so:
%via the following procedure:
\label{p:FM_2^W_k}

\medskip

\begin{itemize}
\item[1.]
Begin with the following disjoint union, where $T$ varies over type-$k$ Getzler--Jones data:
\begin{align}
\label{eq:FM_2^W_definition_no_quotient}
\coprod_T
\wt\GJ_T.
\end{align}

\medskip

\item[2.]
Fix a type-$k$ Getzler--Jones datum $T$, and fix a cell $F$ in the boundary of $\wt\GJ_T = \prod_{v \in T_\inte} W_{\bn(v)}^W$.
$F$ lies inside a product of cells in the 2-associahedra that comprise $\wt\GJ_T$ --- that is, we may write $F \subset \prod_{v \in T_\inte} F_v \subset \prod_{v \in T_\inte} W_{\bn(v)}^W$, where $F_v$ is a cell in $W_{\bn(v)}^W$.
For every $v$, we have a map $q_v$ from $W_{\bn(v)}^W$ to a product of 2-associahedra; by combining these, we obtain a map from $F$ to a product of 2-associahedra.
In fact, we can regard the target of this map as a Getzler--Jones cell.

\medskip

\item[3.]
We take the quotient of the disjoint union in \eqref{eq:FM_2^W_definition_no_quotient} by attaching the constituent spaces together via the maps we defined in the last step.
\end{itemize}
We define $\FM_2^W(1)$ to be a point.
%\null\hfill$\triangle$
\end{definition}

\noindent
It is a consequence of the simplicial structure of the $\wt\GJ_T$'s that each $\FM_2^W(k)$ has the structure of a CW complex.
As noted above, a result of our definition is that $\FM_2^W(k)$ decomposes as a union of Getzler--Jones cells, over all Getzler--Jones data of type $k$.
%:
%\begin{align}
%\FM_2^W(k)
%=
%\bigcup_T
%\GJ_T.
%\end{align}

\subsection{The operad structure on $\FM_2^W$}

%\begin{lemma}
%To give a CW decomposition of a topological space $X$, it suffices to specify a finite collection $(C_i)$ of subspaces of $X$ that has the following properties:
%\medskip
%\begin{itemize}
%\item[\bf 1.]
%The $C_i$'s disjointly decompose $\FM_k$: $\FM_k = \bigsqcup_i C_i$.
%
%\medskip
%
%\item[\bf 2.]
%For every $i$, there is a homeomorphism $C_i \cong B^{d(C_i)}$.
%
%\medskip
%
%\item[\bf 3.]
%For every $i$, the boundary $\partial C_i$ is contained in the $\bigl(d(C_i)-1\bigr)$-skeleton:
%\begin{align}
%\partial C_i \subset \bigcup_{j:\: d(C_j) < d(C_i)} C_j.
%\end{align}
%\end{itemize}
%\end{lemma}

%\begin{proof}
%\end{proof}

%\subsection{Proof that $\FM_2^W$ is an operad, and that the composition maps are cellular}

\begin{definition}
\label{def:FM_2^W_composition}
Fix $k$, $\ell$, and $i \in [1,k]$.
We wish to define the
%composition
map
\begin{align}
\circ_i
\colon
\FM_2^W(k) \times \FM_2^W(\ell) \to \FM_2^W(k+\ell-1).
\end{align}
To do so, fix Getzler--Jones data $T$ and $T'$ of types $k$ and $\ell$, respectively, and fix cells $F \subset \GJ_T$ and $F' \subset \GJ_{T'}$.
We will define $\circ_i$ on
\begin{align}
\GJ_T \times \GJ_{T'}
=
\prod_{v \in T_\inte \sqcup T'_\inte}
\rm W_{\bn(v)}^W.
\end{align}
Define $T''$ to be the result of grafting $T'$ to the $i$-th leaf of $T$, and completing it to a Getzler--Jones datum in the obvious way.
We define $\circ_i$ on $\GJ_T\times\GJ_{T'}$ to be the identification of $\GJ_T \times \GJ_{T'}$ with $\GJ_{T''}$.
%\null\hfill$\triangle$
\end{definition}

\begin{lemma}
\label{lem:it's_an_operad}
Taken together, the spaces $\bigl(\FM_2^W(k)\bigr)_{k\geq1}$ together with the composition operations $\circ_i$ form a non-$\Sigma$ operad.
\end{lemma}

\begin{proof}
This is immediate from the definition.
\end{proof}

\begin{lemma}
\label{lem:compositions_are_cellular}
The composition maps
\begin{align}
\circ_i
\colon
\FM_2^W(k) \times \FM_2^W(\ell)
\to
\FM_2^W(k+\ell-1)
\end{align}
are cellular.
\end{lemma}

\begin{proof}
Similar to the proof of Prop.~\ref{prop:K_r^W_operad}.
\end{proof}

\begin{remark}
\label{rem:FM_W}
Fix $r \geq 1$, $\bn \in \bZ_{\geq0}^r\setminus\{\bzero\}$, and a 2-permutation $\sigma$ of type $\bn$.
Then the associated forgetful map
\begin{align}
f_\sigma^W\colon W_\bn^W \to \FM_2^W(|\bn|)
\end{align}
is cellular.
This map is defined in the obvious way: we first identify $W_\bn^W$ with the corresponding $\wt\GJ_T$, where $T$ is a Getzler--Jones datum whose associated tree $T$ is a corolla with $|\bn|$ leaves.
Then, we include $\wt\GJ_T$ into the disjoint union $\bigsqcup_T \wt\GJ_T$, and finally take the quotient to land in $\FM_2^W(|\bn|)$.
%\null\hfill$\triangle$
\end{remark}


\begin{thebibliography}{10}

\bibitem[Ab]{abouzaid}
M.~Abouzaid.
\newblock Symplectic cohomology and Viterbo's theorem.
\newblock In {\it Free Loop Spaces in Geometry and Topology}, IRMA Lectures in Mathematics and Theoretical Physics.

\bibitem[Ba]{barber}
D.~Barber.
A Comparison of Models for the Fulton--Macpherson Operads.
\newblock Thesis.

\bibitem[BoaVo]{boardman_vogt}
J.M.~Boardman, R.M.~Vogt.
\newblock
{\it Homotopy invariant algebraic structures on topological spaces}, Lect.\ Notes Math.\ 347 (1973).

\bibitem[Bo1]{b:2-associahedra}
N.~Bottman.
\newblock 2-associahedra.
\newblock {\it Algebra.\ Geom.\ Topol.} 19 (2019), no.\ 2, 743--806.

\bibitem[Bo2]{b:realization}
N.~Bottman.
\newblock Moduli spaces of witch curves topologically realize the 2-associahedra.
\newblock Accepted (2019), {\it Journal of Symplectic Geometry}.

\bibitem[Bo3]{b:sing}
N.~Bottman.
\newblock Pseudoholomorphic quilts with figure eight singularity.
\newblock Accepted (2018), {\it Journal of Symplectic Geometry}.

\bibitem[Bo4]{b:thesis}
N.~Bottman.
\newblock Pseudoholomorphic quilts with figure eight singularity.
\newblock Thesis (2015).

\bibitem[BoCa]{bc}
N.~Bottman, S.~Carmeli.
\newblock $(A_\infty,2)$-categories and relative 2-operads.
\newblock Submitted; available on \url{https://arxiv.org}.

\bibitem[BoOb]{bo}
N.~Bottman, A.~Oblomkov.
\newblock A compactification of the moduli space of marked vertical lines in $\bC^2$.
\newblock Submitted; available at \url{https://arxiv.org/abs/1910.02037}.

\bibitem[BoWe]{bw:compactness}
N.~Bottman, K.~Wehrheim.
\newblock Gromov compactness for squiggly strip shrinking in pseudoholomorphic quilts.
\newblock {\it Selecta Math. (N.S.)} 24 (2018), no.\ 4, pp.\ 3381--3443.

\bibitem[FoxNeu]{fox_neuwirth}
R.~Fox, L.~Neuwirth.
\newblock The braid groups.
\newblock {\it Math.\ Scand.} 10 (1962), 119--126.

\bibitem[FriPi]{fp}
R.~Fritsch, R.A.~Piccinini.
\newblock {\it Cellular structures in topology.}
\newblock Cambridge Studies in Advanced Mathematics, 19.
Cambridge University Press, Cambridge, 1990.

\bibitem[FuMac]{fm}
W.~Fulton, R.~MacPherson.
\newblock A compactification of configuration spaces.
\newblock {\it Ann.\ of Math.} (2), 139(1):183--225, 1994.

\bibitem[GeJo]{getzler_jones}
E.~Getzler, J.D.S.~Jones.
\newblock Operads, homotopy algebra and iterated integrals for double loop spaces.
\newblock Preprint (1994), available at \url{https://arxiv.org/abs/hep-th/9403055}.

\bibitem[Ma'uWeWo]{mww}
S.~Ma'u, K.~Wehrheim, C.~Woodward.
\newblock $A_\infty$ functors for Lagrangian correspondences.
\newblock {\it Selecta Math.\ (N.S.)} 24 (2018), no.\ 3, 1913--2002.

\bibitem[Sa1]{salvatore_main}
P.~Salvatore.
\newblock A cell decomposition of the Fulton MacPherson operad.
\newblock Preprint (2019), available at \url{https://arxiv.org/abs/1906.07694}.

\bibitem[Sa2]{salvatore_aux}
P.~Salvatore.
\newblock The Fulton MacPherson operad and the W-construction.
\newblock Preprint (2019), available at \url{https://arxiv.org/abs/1906.07696}.

\bibitem[Vo]{voronov}
A.A.~Voronov.
\newblock Homotopy Gerstenhaber algebras.
\newblock In {\it Conf\'erence Mosh\'e Flato 1999}.
Mathematical Physics Studies, vol.\ 21/22.
Springer, Dordrecht.

\end{thebibliography}
\end{document}